\documentclass[11pt]{amsart}
\usepackage{hyperref}
\usepackage[capitalise, nameinlink]{cleveref}
\usepackage[all]{xy}
\makeatletter
\@namedef{subjclassname@2010}{%
  \textup{2010} Mathematics Subject Classification}
\makeatother
\newtheorem{theorem}{Theorem}[section]

\newtheorem{lem}[theorem]{Lemma}
\newtheorem{mainthm}{Theorem}

\theoremstyle{definition}
\newtheorem{defn}[theorem]{Definition}
\newtheorem{rem}[theorem]{Remark}

\def\R{{\mathbb{R}}}

\newcommand{\U}{\widehat{\mathcal{U}}}
\newcommand{\G}{\widehat{\mathcal{GL}}}

\hypersetup{
	colorlinks,
	linkcolor=blue,          
	filecolor=magenta,      
	urlcolor=cyan           
}

\begin{document}
\title[K-stability of continuous $C(X)$-algebras]{K-stability of continuous $C(X)$-algebras}
\author[Apurva Seth, Prahlad Vaidyanathan]{Apurva Seth, Prahlad Vaidyanathan}
\address{Department of Mathematics\\ Indian Institute of Science Education and Research Bhopal\\ Bhopal ByPass Road, Bhauri, Bhopal 462066\\ Madhya Pradesh. India.}
\email{apurva17@iiserb.ac.in, prahlad@iiserb.ac.in}
\date{}
\begin{abstract}
A C*-algebra is said to be K-stable if its nonstable K-groups are naturally isomorphic to the usual K-theory groups. We study continuous $C(X)$-algebras, each of whose fibers are K-stable. We show that such an algebra is itself K-stable under the assumption that the underlying space $X$ is compact, metrizable, and of finite covering dimension.
\end{abstract}
\maketitle
\parindent 0pt

Nonstable K-theory is the study of the homotopy groups of the unitary group of a C*-algebra. The study of these groups was initiated by Rieffel \cite{rieffel2}, who showed that, for an irrational rotation algebra $A$, the inclusion map from $A$ to $M_n(A)$ induces an isomorphism between the corresponding homotopy groups. In other words, the nonstable $K$-groups are naturally isomorphic to the usual $K$-theory groups of the algebra. \\

The theory was further explored by Thomsen \cite{thomsen}, who used the notion of quasi-unitaries to profitably extend nonstable K-theory to encompass non-unital C*-algebras. Furthermore, he showed that this forms a homology theory, which allowed him to explicitly calculate these groups for certain C*-algebras. In particular, he showed that certain infinite dimensional C*-algebras (including the Cuntz algebras and simple infinite dimensional AF-algebras) satisfy the property enjoyed by the irrational rotation algebra mentioned above; a property he termed \emph{$K$-stability}. \\

Since then, it has been proved (See \cref{subsec:nonstable_k_theory}) that a variety of interesting simple C*-algebras are $K$-stable. The goal of this paper is to enlarge this class of C*-algebras to include non-simple C*-algebras. \\

By the Dauns-Hoffmann theorem (See \cite{nilsen} or \cite{rieffel3}), any non-simple C*-algebra may be represented as the section algebra of an upper semi-continuous C*-bundle over a compact space. If we assume that the underlying space $X$ is Hausdorff, then such an algebra carries a non-degenerate, central action of $C(X)$, and is called a $C(X)$-algebra. An interesting sub-class of $C(X)$-algebras are ones that come equipped with a natural continuity condition. These algebras, called \emph{continuous $C(X)$-algebras}, are particularly tractable as one can often take phenomena that occur at each fiber and propagate them to understand local behaviour of the algebra. Using a compactness argument, one may even be able to understand global behaviour. It is this idea that we employ in this paper to prove our main theorem.

\begin{mainthm}\label{thm:mainthm}
Let $X$ be a compact metric space of finite covering dimension, and let $A$ be a continuous $C(X)$-algebra. If each fiber of $A$ is $K$-stable, then $A$ is K-stable.
\end{mainthm}


\section{Preliminaries}

\subsection{Nonstable $K$-theory}\label{subsec:nonstable_k_theory}

We begin by reviewing the work of Thomsen of constructing the nonstable $K$-groups associated to a C*-algebra. For the proofs of all the facts mentioned below, the reader may refer to \cite{thomsen}. \\

Let $A$ be a $C^*$-algebra (not necessarily unital). Define an associative composition $\cdot$ on $A$ by
\begin{equation}\label{eqn:composition}
a\cdot b=a+b-ab
\end{equation}
Henceforth, if $a,b\in A$, then $ab$ will denote the usual multiplication in the algebra, and should not be confused with $a\cdot b$, which will denote the above composition. \\

An element $a\in A$ is said to be quasi-invertible if there exists $b\in A$ such that $a\cdot b = b\cdot a = 0$, and we write $\G(A)$ for the set of all quasi-invertible elements in $A$. An element $u\in A$ is said to be a quasi-unitary if $u\cdot u^{\ast} = u^{\ast}\cdot u = 0$, and we write $\U(A)$ for the set of all quasi-unitary elements in $A$. \\

If $B$ is a unital $C^*$-algebra, we write $GL(B)$ for the group of invertibles in $B$ and $U(B)$ for the group of unitaries in $B$. Let $A^+$ denote the unitization of $A$. It follows by \cite[Lemma 1.2]{thomsen} that an element $a\in A$ is quasi-invertible if and only if $1-a \in GL(A^+)$; and similarly, $u\in A$ is a quasi-unitary if and only if $(1-u) \in U(A^+)$. Therefore, $\G(A)$ is open in $A$, $\U(A)$ is closed in $A$, and they both form topological groups. Furthermore, the map $r: \G(A)\to \U(A)$ given by 
\[
r(a) := 1-(1-a)((1-a^*)(1-a))^{-1/2}
\]
is a strong deformation retract, and hence a homotopy equivalence. \\
%



For elements $u,v\in \U(A)$, we write $u\sim v$ if there is a continuous function $f:[0,1]\to \U(A)$ such that $f(0) = u$ and $f(1) = v$. We write $\U_0(A)$ for the set of $u\in \U(A)$ such that $u\sim 0$. The next result, which we will use repeatedly throughout the paper, follows from \cite[Theorem 1.9]{thomsen} and \cite[Theorem 4.8]{dold}.

\begin{theorem}\label{thm:fibration}
If $\varphi : A\to B$ is a surjective $\ast$-homomorphism between two C*-algebras, then the induced map $\varphi : \U_0(A)\to \U_0(B)$ is a Serre fibration.
\end{theorem}

%
%

For a C*-algebra $A$, the suspension of $A$ is defined to be $SA:=  C_0(\mathbb{R})\otimes A$. For $n>1$, we set $S^nA:= S(S^{n-1}A)$. We then have

\begin{lem}\label{lem:suspension_adjoint}\cite[Lemma 2.3]{thomsen}
For any C*-algebra $A$, $\pi_n(\U(A)) \cong \pi_0(\U(S^nA))$.
\end{lem}

\begin{defn}
The nonstable $K$-groups of a C*-algebra $A$ are defined as
\[
k_n(A):= \pi_{n+1}(\U(A)), \qquad\text{ for } n = -1, 0,\ldots
\]
\end{defn}

By \cref{lem:suspension_adjoint} and Bott periodicity, it follows \cite[Proposition 2.6]{thomsen} that, if $A$ is a stable C*-algebra, then $k_n(A)\cong K_n(A)$, where $K_n(A)$ denotes the usual $K$-theory groups of $A$. Motivated by this, Thomsen defines the notion of $K$-stability of a C*-algebra.

\begin{defn}
Let $A$ be a $C^*$-algebra and $m\geq 2$. Define $\iota_m: M_{m-1}(A)\to M_m(A)$ by
\[
a\mapsto\begin{pmatrix}
a&0\\
0&0
\end{pmatrix}
\]
We say that $A$ is $K$-stable if $(\iota_m)_{\ast}: k_n(M_{m-1}(A))\to k_n(M_m(A))$ is an isomorphism for all $n=-1,0,1,2\hdots$ and all $m=2,3,4\hdots$.
\end{defn}

\begin{rem}\label{rem:k_stable_examples}
The following C*-algebras are known to be $K$-stable:
\begin{itemize}
\item If $\mathcal{Z}$ denotes the Jiang-Su algebra, then $A\otimes \mathcal{Z}$ is $K$-stable for any C*-algebra $A$ \cite{jiang}. In particular, every separable, approximately divisible C*-algebra is $K$-stable \cite{toms_winter}.
\item Every irrational rotation algebra is $K$-stable \cite{rieffel2}.
\item If $\mathcal{O}_n$ denotes the Cuntz algebra, then $A\otimes \mathcal{O}_n$ is $K$-stable for any C*-algebra $A$ \cite{thomsen}.
\item If $A$ is an infinite dimensional simple AF-algebra, then $A\otimes B$ is $K$-stable for any C*-algebra $B$ \cite{thomsen} .
\item If $A$ is a purely infinite, simple C*-algebra, and $p$ any non-zero projection of $A$, then $pAp$ is $K$-stable \cite{zhang}.
\end{itemize}
Note that some of the examples mentioned above may be subsumed into the first example as they are known to absorb $\mathcal{Z}$ tensorially. However, it is worth mentioning that the original proofs of $K$-stability for these algebras does not use $\mathcal{Z}$-stability. Furthermore, in the case where the algebras are $\mathcal{Z}$-stable, \cref{thm:mainthm} may be deduced from \cite[Theorem 4.6]{hirshberg_et_al}.
\end{rem}


We conclude this section with an important observation about $K$-stable C*-algebras. 

\begin{lem}\label{lem:k_stable_def_retract}
If $A$ is $K$-stable, then for any $m\geq 2, \iota_m(\U(M_{m-1}(A))$ is a strong deformation retract of $\U(M_m(A))$.
\end{lem}
\begin{proof}
Note that $\G(A)$ is an absolute neighbourhood retract \cite[Theorem 5]{palais}, and therefore, the pair $(\G(M_m(A)), \iota_m(\G(M_{m-1}(A))))$ has the homotopy extension property with respect to all spaces \cite[Theorem 7]{palais}. If $A$ is $K$-stable, then $\iota_m: \G(M_{m-1}(A))\to \G(M_m(A))$ is a weak homotopy equivalence. However, since $\G(A)$ is an open subset of a normed linear space, $\G(A)$ has the homotopy type of a CW-complex \cite[Chapter IV, Corollary 5.5]{lundell}. By Whitehead's theorem, it follows that $\iota_m$ is a homotopy equivalence, so $\iota_m(\G(M_{m-1}(A))$ is a strong deformation retract of $\G(M_m(A))$ by \cite[Theorem 0.20]{hatcher}. Since the retractions $r : \G(M_k(A))\to \U(M_k(A))$ commute with the inclusion map $\iota_k$, we conclude that $\iota_m(\U(M_{m-1}(A))$ is a strong deformation retract of $\U(M_m(A))$.
\end{proof}

\subsection{$C(X)$-algebras}
Let $A$ be a C*-algebra, and $X$ a compact Hausdorff space. We say that $A$ is a $C(X)$-algebra \cite[Definition 1.5]{kasparov} if there is a unital $\ast$-homomorphism $\theta : C(X)\to ZM(A)$, where $ZM(A)$ denotes the center of the multiplier algebra of $A$. \\


If $Y\subset X$ is closed, the set $C_0(X,Y)$ of functions in $C(X)$ that vanish on $Y$ is a closed ideal of $C(X)$. By the Cohen factorization theorem \cite[Theorem 4.6.4]{brown_ozawa}, $C_0(X,Y)A$ is a closed, two-sided ideal of $A$. The quotient of $A$ by this ideal is denoted by $A(Y)$, and we write $\pi_Y : A\to A(Y)$ for the quotient map (also referred to as the restriction map). If $Z\subset Y$ is a closed subset of $Y$, we write $\pi^Y_Z : A(Y)\to A(Z)$ for the natural restriction map, so that $\pi_Z = \pi^Y_Z\circ \pi_Y$. If $Y = \{x\}$ is a singleton, we write $A(x)$ for $A(\{x\})$ and $\pi_x$ for $\pi_{\{x\}}$. The algebra $A(x)$ is called the fiber of $A$ at $x$. For $a\in A$, write $a(x)$ for $\pi_x(a)$. For each $a\in A$, we have a map $\Gamma_a : X\to \R$ given by $x \mapsto \|a(x)\|$. We say that $A$ is a continuous $C(X)$-algebra if $\Gamma_a$ is continuous for each $a\in A$. \\

If $A$ is a continuous $C(X)$-algebra, we will often have reason to consider other $C(X)$-algebras obtained from $A$. At that time, the following result of Kirchberg and Wasserman will be useful.

\begin{theorem}\cite[Remark 2.6]{kw_fields}\label{thm:kw_tensor}
Let $X$ be a compact Hausdorff space, and let $A$ be a continuous $C(X)$-algebra. If $B$ is a nuclear C*-algebra, then $A\otimes B$ is a continuous $C(X)$-algebra whose fiber at a point $x\in X$ is $A(x)\otimes B$.
\end{theorem}

Finally, one fact that plays a crucial role in our investigation is that a $C(X)$-algebra may be patched together from quotients in the following way: Let $B,C,$ and $D$ be C*-algebras, and $\delta : B\to D$ and $\gamma : C\to D$ be $\ast$-homomorphisms. We define the pullback of this system to be
\[
A = B\oplus_D C := \{(b,c) \in B\oplus C : \delta(b) = \gamma(c)\}
\]
This is describe by a diagram
\begin{equation*}
\xymatrix{
A\ar[r]^{\phi}\ar[d]_{\psi} & B\ar[d]^{\delta} \\
C\ar[r]^{\gamma} & D
}
\end{equation*}
where $\phi(b,c) = b$ and $\psi(b,c) = c$.

\begin{lem}\cite[Lemma 2.4]{mdd_finite}\label{lem:cx_algebra_pullback}
Let $X$ be a compact Hausdorff space and $Y$ and $Z$ be two closed subsets of $X$ such that $X = Y\cup Z$. If $A$ is a $C(X)$-algebra, then $A$ is isomorphic to the pullback
\[
\xymatrix{
A\ar[r]^{\pi_Y}\ar[d]_{\pi_Z} & A(Y)\ar[d]^{\pi^Y_{Y\cap Z}} \\
A(Z)\ar[r]^{\pi^Z_{Y\cap Z}} & A(Y\cap Z)
}
\]
\end{lem}

\subsection{Notational Conventions} We fix some notational conventions we will use repeatedly: If $A$ is a C*-algebra, we write $\iota_A : A\to M_2(A)$ for the natural inclusion map. When there is no ambiguity, we write $\iota$ for this map. Moreover, if $\varphi:A\to B$ is a $\ast$-homomorphism between two C*-algebras, then the induced map from $\U(A)$ to $\U(B)$ is also denoted by $\varphi$. \\

If $A$ is a continuous $C(X)$-algebra, then $M_2(A)$ is also a continuous $C(X)$-algebra with fibers $M_2(A(x))$ by \cref{thm:kw_tensor}. We will often consider both simultaneously, so we fix the following convention: If $Y\subset X$ is a closed set, we denote the restriction map by $\eta_Y : M_2(A)\to M_2(A(Y))$, and write $\iota_Y : A(Y)\to M_2(A(Y))$ for the natural inclusion map. If $Y = X$, we simply write $\iota$ for $\iota_X$. Note that $\eta_Y\circ \iota = \iota_Y\circ \pi_Y$. Once again, if $Y = \{x\}$, we simply write $\iota_x$ for $\iota_{\{x\}}$. \\


Finally, suppose $f$ and $g$ are two continuous paths in a topological space $Y$. If $f(1) = g(0)$, we write $f\bullet g$ for the concatenation of the two paths. If $f$ and $g$ agree at end-points, we write $f\sim_h g$ if there is a path homotopy between them. Furthermore, we write $\overline{f}$ for the path $\overline{f}(t) := f(1-t)$. If $Y = \U(A)$ for some C*-algebra $A$, we write $f\cdot g$ for the path $t\mapsto f(t)\cdot g(t)$, and we write $f^{\ast}$ for the path $t\mapsto f(t)^{\ast}$.

\section{Main Results}


\begin{lem}\label{lem:tensor_commutative}
Let $A$ be a $K$-stable C*-algebra and $X$ a locally compact Hausdorff space, then $C_0(X)\otimes A$ is $K$-stable.
\end{lem}
\begin{proof}
Suppose first that $X$ is compact. For simplicity of notation, we write $B$ for $C(X)\otimes A$. By \cref{lem:k_stable_def_retract}, for each $m\geq 2$, there is a retraction $r: \U(M_m(A)) \to \U(M_{m-1}(A))$ and a homotopy $F:[0,1]\times \U(M_m(A)) \to \U(M_m(A))$ such that $F(0,v) = v$ and $F(1,v)= r(v)$ for all $v\in \U(M_m(A))$. Clearly, we may identify $\U(M_m(B))$ with $C(X,\U(M_m(A))$, where the latter denotes the space of continuous functions from $X$ to $\U(M_m(A))$, equipped with the uniform topology (which coincides with the compact-open topology). Therefore, we may define $\widetilde{r} : \U(M_m(B)) \to \U(M_{m-1}(B))$ by
\[
\widetilde{r}(f)(x) := r(f(x))
\]
and $\widetilde{F} : [0,1]\times \U(M_m(B))\to \U(M_m(B))$ by
\[
\widetilde{F}(s,f)(x) := F(s,f(x))
\]
It is easy to see that $\widetilde{r}$ is a retraction, and that $\widetilde{F}$ is continuous, and implements a homotopy between $\iota_{M_{m-1}(B)}\circ \widetilde{r}$ and $\text{id}_{\U(M_m(B))}$. Thus, $B$ is $K$-stable. \\

If $X$ is not compact, let $X^+$ denote its one-point compactification. We now have a short exact sequence $0\to C_0(X)\otimes A\to C(X^+)\otimes A\to A \to 0$, which induces a long exact sequence of $k$-groups by \cite[Theorem 2.5]{thomsen}. By the first part of the argument, $C(X^+)\otimes A$ is $K$-stable, so the result follows from the five lemma.
\end{proof}

Let $A$ be a C*-algebra, and $c\in A$ a self-adjoint element. We define
\[
\Lambda(c) := - \sum_{n=1}^{\infty} \frac{(ic)^n}{n!}
\]
Observe that, in $A^+, \Lambda(c) = 1 - \exp(ic)$, so that $\Lambda(c) \sim 0$ in $\U(A)$ via the path $t \mapsto \Lambda(tc)$. The next lemma is implicit in \cite[Lemma 1.7]{thomsen}, but we spell it out since its proof is crucial to us.

\begin{lem}\label{lem:quasi_homotopy}
Let $a,b\in \U(A)$ such that $\|a-b\| < 2$, then $a\sim b$ in $\U(A)$
\end{lem}
\begin{proof}
Consider $A$ as an ideal in $A^+$. If $\|a-b\| < 2$, then $d := a\cdot b^{\ast}$ satisfies
\[
\|d\| = \|1-((1-a)(1-b^{\ast}))\| = \|((1-b) - (1-a))(1-b^{\ast})\| \leq \|a-b\| < 2
\]
Hence, $(1-d)$ is a unitary in $A^+$, whose spectrum does not contain $1$. Therefore, there is a continuous function $g : S^1 \to \R$ such that $g(1) = 0$ and $\exp(ig(x)) = x$ for all $x\in \sigma(1-d)$. Define $c := g(1-d)$, then $c$ is self-adjoint and $a\cdot b^{\ast} = \Lambda(c)$. Thus $a = \Lambda(c)\cdot b\sim b$ in $\U(A)$.
\end{proof}

The next lemma is a variant of \cite[Exercise 2.8]{rordam} for quasi-unitaries.

\begin{lem}\label{lem:quasi_close}
Given $\epsilon > 0$, there is a $\delta > 0$ with the following property: If $A$ is a C*-algebra and $a\in A$ such that $\|a\cdot a^{\ast}\| < \delta$ and $\|a^{\ast}\cdot a\| < \delta$, then there is a quasi-unitary $u \in \U(A)$ such that $\|a-u\| < \epsilon$.
\end{lem}
\begin{proof}
Fix $0< \delta < \min\{\epsilon,1\}$, and consider $A$ as an ideal in $A^+$. If $b:= (1-a)$, the hypothesis implies that $\|1 - bb^{\ast}\| = \|a\cdot a^{\ast}\| < \delta < 1$, and $\|1-b^{\ast}b\| < 1$ as well. Hence, $bb^{\ast}$ and $b^{\ast}b$ are both invertible, and therefore $b$ and $b^{\ast}$ are also invertible. Thus, $|b| = (b^{\ast}b)^{1/2} \in GL(A^+)$, and $v := b|b|^{-1} \in U(A^+)$. Furthermore, $\sigma(b^{\ast}b) \subset (1-\delta,1+\delta)$, so $\sigma(|b|) \subset ((1-\delta)^{1/2},(1+\delta)^{1/2}) \subset (1-\delta,1+\delta)$. Hence, $\||b|-1\| < \delta$, so that $u:= 1-v \in \U(A)$, and
\[
\|u-a\| = \|v - b\| = \|v(1 - |b|)\| \leq \|1-|b|\| < \delta < \epsilon
\]
as required.
\end{proof}

\begin{lem}\label{lem:local_injective}
Let $X$ be a compact, Hausdorff space and let $A$ be a continuous $C(X)$-algebra. Let $a\in \U(A)$ and $x\in X$ such that $a(x) \sim 0$ in $\U(A(x))$. Then there is a closed neighbourhood $Y$ of $x$ such that $\pi_Y(a) \sim 0$ in $\U(A(Y))$.
\end{lem}
\begin{proof}
Since $\pi_x : \U_0(A) \to \U_0(A(x))$ is a fibration, there is a path $H :[0,1]\to \U(A)$ such that $H(0) = 0$ and $H(1)(x) = a(x)$. Let $b := H(1)$, then since $A$ is a continuous $C(X)$-algebra, there is a closed neighbourhood $Y$ of $x$ such that $\|\pi_Y(a-b)\| < 2$. It follows by \cref{lem:quasi_close} that $\pi_Y(a) \sim \pi_Y(b)$ in $\U(A(Y))$. Furthermore, the path $G : [0,1]\to \U(A(Y))$ is of the form
\[
G(t) = \pi_Y(b)\cdot \Lambda(t\pi_Y(c))
\]
for some self-adjoint element $c \in A$ (since every self-adjoint element in $A(Y)$ lifts to a self-adjoint element in $A$). Since $a(x) = b(x)$, the proof of \cref{lem:quasi_homotopy} in fact ensures that we may choose $c$ such that $c(x) = 0$. Therefore, $\pi_x\circ G(t) = a(x)$ for all $t\in [0,1]$. Concatenating the paths $\pi_Y\circ H$ and $G$, we obtain a path connecting $0$ to $\pi_Y(a)$ in $\U(A(Y))$.
\end{proof}

Our proof of \cref{thm:mainthm} is by induction on the covering dimension of the underlying space. The next theorem is the base case, and it holds even if the space is not metrizable.

\begin{theorem}\label{thm:dim_zero_case}
Let $X$ be a compact Hausdorff space of zero covering dimension, and let $A$ be a continuous $C(X)$-algebra. If each fiber of $A$ is $K$-stable, then so is $A$.
\end{theorem}
\begin{proof}
If $A$ is a continuous $C(X)$-algebra, then so is every suspension of $A$. Furthermore, $(S^nA)(x) \cong S^n(A(x))$ by \cref{thm:kw_tensor}, and $S^n(A(x))$ is $K$-stable by \cref{lem:tensor_commutative}. Hence, by \cref{lem:suspension_adjoint}, it suffices to show that, for each $m\geq 2$, the map
\[
(\iota_m)_{\ast} : \pi_0(M_{m-1}(A)) \to \pi_0(M_m(A))
\]
is an isomorphism. However, by \cref{thm:kw_tensor}, each $M_n(A)$ is also a continuous $C(X)$-algebra, with fibers $M_n(A(x))$, which is $K$-stable if $A(x)$ is $K$-stable. Therefore,  suffices to show that the map
\[
\iota_{\ast} : \pi_0(\U(A)) \to \pi_0(\U(M_n(A))
\]
is an isomorphism for each $n\geq 2$. For simplicity of notation, we fix $n=2$.\\

We first consider injectivity. Suppose $a\in \U(A)$ such that $\iota(a) \sim 0$ in $\U(M_2(A))$. Then, for any $x\in X$, $\iota_x(a(x)) \sim 0 \text{ in } \U(M_2(A(x))$. Since $A(x)$ is $K$-stable, $a(x) \sim 0$ in $\U(A(x))$. By \cref{lem:local_injective}, there is a closed neighbourhood $Y_x$ of $x$ such that $\pi_{Y_x}(a) \sim 0$ in $\U(A(Y_x))$. Since $X$ is compact and zero dimensional, we obtain disjoint open sets $\{Y_{x_1}, Y_{x_2}, \ldots, Y_{x_n}\}$ which cover $X$. Then by \cref{lem:cx_algebra_pullback}, 
\[
A \cong A(Y_{x_1})\oplus A(Y_{x_2})\oplus \ldots \oplus A(Y_{x_n})
\]
via the map $b\mapsto (\pi_{Y_{x_1}}(b), \pi_{Y_{x_2}}(b), \ldots, \pi_{Y_{x_n}}(b))$. Since $\pi_{Y_{x_i}}(a) \sim 0$ for each $1\leq i\leq n$, it follows that $a\sim 0$, so $\iota_{\ast}$ is injective. \\

For surjectivity, choose $u \in \U(M_2(A))$, and we wish to construct a quasi-unitary $\omega\in \U(A)$ such that $u\sim \iota(\omega)$. To this end, fix $x\in X$. Then $u(x) \in \U(M_2(A(x))$. Since $A(x)$ is $K$-stable, there exists $f_x \in \U(A(x))$ and a path $g_x : [0,1]\to \U(M_2(A(x))$ such that $g_x(0)  = u(x)$ and $g_x(1) = \iota_x(f_x)$. Choose $e_x \in A$ such that $e_x(x) = f_x$ (Note that $e_x$ may not be a quasi-unitary). \\

Since the map $\eta_x : \U_0(M_2(A)) \to \U_0(M_2(A(x)))$ is a fibration, $g_x$ lifts to a path $G_x : [0,1]\to \U(M_2(A))$ such that $G_x(0) = u$. Let $b_x := G_x(1)$, and  so that $b_x(x) = \iota_x(e_x(x))$. Choose $\delta > 0$ so that conclusion of \cref{lem:quasi_close} holds for $\epsilon = 1$. Since $A$ is a continuous $C(X)$-algebra, there is a closed neighbourhood $Y_x$ of $x$ such that
\[
\|\eta_{Y_x}(b_x) - \eta_{Y_x}(\iota(e_x))\| < 1, 
\|\pi_{Y_x}(e_x^{\ast}\cdot e_x)\| < \delta, \text{ and}
\|\pi_{Y_x}(e_x\cdot e_x^{\ast})\| < \delta
\]
By \cref{lem:quasi_close}, there is a quasi-unitary $d_x \in \U(A(Y_x))$ such that $\|d_x - \pi_{Y_x}(e_x)\| < 1$, so that $\|\iota_{Y_x}(d_x) - \eta_{Y_x}(b_x)\| < 2$. By \cref{lem:quasi_homotopy}, $\iota_{Y_x}(d_x) \sim \eta_{Y_x}(b_x)$ in $\U(A(Y_x))$. Hence, $\iota_{Y_x}(d_x) \sim \eta_{Y_x}(u)$. As before, since $X$ is compact and zero-dimensional, we may choose disjoint, open sets $\{Y_{x_1}, Y_{x_2},\ldots, Y_{x_n}\}$ so that
\[
A \cong A(Y_{x_1})\oplus A(Y_{x_2})\oplus \ldots \oplus A(Y_{x_n})
\]
via the map $a\mapsto (\pi_{Y_{x_1}}(a), \pi_{Y_{x_2}}(a), \ldots, \pi_{Y_{x_n}}(a))$. Similarly,
\[
M_2(A) \cong M_2(A(Y_{x_1}))\oplus M_2(A(Y_{x_2}))\oplus \ldots \oplus M_2(A(Y_{x_n}))
\]
via the map $b \mapsto (\eta_{Y_{x_1}}(b), \eta_{Y_{x_2}}(b),\ldots, \eta_{Y_{x_n}}(b))$. Therefore, there exists $\omega\in \U(A)$ such that $\pi_{Y_{x_i}}(\omega) = d_{x_i}$ for all $1\leq i\leq n$. Furthermore, for each $1\leq i\leq n, \eta_{Y_{x_i}}(\iota(\omega)) = \iota_{Y_{x_i}}(d_{x_i}) \sim \eta_{Y_{x_i}}(u)$ in $\U(A(Y_{x_i}))$, so that $\iota(\omega) \sim u$ in $\U(A)$, as required.
\end{proof}

The next two lemmas help us to extend the above argument to higher dimensional spaces.

\begin{lem}\label{lem:path_homotopy}
Let $X$ be a compact Hausdorff space, and $A$ be a continuous $C(X)$-algebra. Let $f_i :[0,1]\to \U(A), i=1,2$ be two paths, and $x\in X$ be a point such that $\pi_x\circ f_1 = \pi_x\circ f_2$.
\begin{enumerate}
\item There is a closed neighbourhood $Y$ of $x$ and a homotopy $H : [0,1]\times [0,1]\to \U(A(Y))$ such that $H(0) = \pi_Y\circ f_1$ and $H(1) = \pi_Y\circ f_2$.
\item If, in addition, $f_1(0) = f_2(0)$ and $f_1(1) = f_2(1)$, then the homotopy in part (1) may be chosen to be a path homotopy.
\end{enumerate}
\end{lem}
\begin{proof}
For the first part, note that $C[0,1]\otimes A$ is itself a continuous $C(X)$-algebra by \cref{thm:kw_tensor}. Hence, there is a closed neighbourhood $Y$ of $x$ such that $\|\pi_Y\circ f_1 - \pi_Y\circ f_2\| < 2$. The result now follows from \cref{lem:quasi_homotopy}. \\

For the second part, an examination of \cref{lem:quasi_homotopy} shows that the homotopy is implemented by a path $H:[0,1]\times [0,1]\to \U(A(Y))$ given by
\[
H(s,t) = \pi_Y(f_2(t))\cdot \Lambda(sh(t))
\]
where $h(t) = g(1-\pi_Y(f_1)(t)\cdot \pi_Y(f_2)(t)^{\ast})$ for an appropriate branch $g:S^1\to \R$ of the log function such that $g(1) = 0$. Since $f_1$ and $f_2$ agree at end-points, it follows that $h(0) = h(1) = 0$. Hence, $H$ is a path homotopy.
\end{proof}

\begin{lem}\label{lem:local_path}
Let $X$ be a compact Hausdorff space, $A$ be a continuous $C(X)$-algebra, and $x\in X$ be a point such that $A(x)$ is $K$-stable. Let $a\in \U(A)$ be a quasi-unitary and $F : [0,1]\to \U(M_2(A))$ be a path such that
\[
F(0) = 0 \text{ and } F(1) = \iota(a)
\]
Then, there is a closed neighbourhood $Y$ of $x$ and a path $L_Y : [0,1]\to \U(A(Y))$ such that
\[
L_Y(0) = 0, \quad L_Y(1) = \pi_Y(a)
\]
and $\iota_Y\circ L_Y$ is path homotopic to $\eta_Y\circ F$ in $\U(M_2(A(Y)))$.
\end{lem}
\begin{proof}
Let $\iota_x : \U(A(x)) \to \U(M_2(A(x))$ be the natural inclusion map. Since $A(x)$ is $K$-stable, \cref{lem:k_stable_def_retract} implies that there is a continuous function $r_x : \U(M_2(A(x))) \to \U(A(x))$ such that $\iota_x\circ r_x \sim \text{id}_{\U(M_2(A(x)))}$. Furthermore, since $r_x$ is a retract, the function $F' := r_x\circ\eta_x\circ F$ is a path in $\U(A(x))$ such that $F'(0) = 0$ and $F'(1) = a(x)$. Consider the commutative diagram
\[
\xymatrix{
\U(A)\ar[r]^{\iota_A}\ar[d]_{\pi_x} & \U(M_2(A))\ar[d]^{\eta_x} \\
\U(A(x))\ar@<1.ex>[r]^{\iota_x} & \U(M_2(A(x))\ar[l]^{r_x}
}
\]
The map $\pi_x : \U_0(A)\to \U_0(A(x))$ is a fibration, so $F'$ lifts to a path $G_x : [0,1]\to \U(A)$ such that $G_x(0) = 0$. Set $b_x := G_x(1)$, then $\iota_A\circ G_x : [0,1]\to \U(M_2(A))$ is a path such that $\iota_A\circ G_x(0) = 0$ and $\iota_A\circ G_x(1) = \iota_A(b_x)$. Furthermore,
\[
\eta_x\circ \iota_A\circ G_x = \iota_x\circ \pi_x\circ G_x = \iota_x\circ F' = \iota_x\circ r_x\circ\eta_x\circ F
\]
Since $\iota_x(\U(A(x))$ is a strong deformation retract of $\U(M_2(A(x)))$, it follows that there is a path homotopy $\widetilde{H} : [0,1]\times [0,1]\to \U(M_2(A(x)))$ such that, for all $s,t\in [0,1]$,
\[
\widetilde{H}(0,t) = \eta_x\circ F(t), \widetilde{H}(1,t) = \eta_x\circ \iota_A\circ G_x(t), \widetilde{H}(s,0) = 0 ,\widetilde{H}(s,1) = \iota_x(a(x))
\]
The map $\eta_x : \U_0(M_2(A))\to \U_0(M_2(A(x))$ is a fibration, and $\eta_x\circ F$ has a lift, so $\widetilde{H}$ lifts to a homotopy $H_1 : [0,1]\times [0,1]\to \U(M_2(A))$ such that $H_1(0,t) = F(t)$, and $\eta_x\circ H_1 = \widetilde{H}$. Similarly, $\eta_x\circ \iota_A\circ G_x$ lifts to $\iota_A\circ G_x$, so $\widetilde{H}$ lifts to a homotopy $H_2 : [0,1]\times [0,1]\to \U(M_2(A))$ such that $H_2(1,t) = \iota_A\circ G_x(t)$, and $\eta_x\circ H_2 = \widetilde{H}$. Note that
\[
\eta_x\circ F(t) = \eta_x\circ H_2(0,t) \text{ and } \eta_x\circ H_1(1,t) = \eta_x\circ \iota_A\circ G_x(t)
\]
Therefore, by \cref{lem:local_injective} applied to the continuous $C(X)$-algebra $C[0,1]\otimes A$, there is a closed neighbourhood $Y'$ of $x$ and a homotopy $H:[0,1]\times [0,1]\to \U(M_2(A(Y'))$ such that
\[
H(0,t) = \eta_{Y'}\circ F(t), H(1,t) = \eta_{Y'}\circ \iota_A\circ G_x(t), \text{ and } \eta^{Y'}_x\circ H(s,1) = \iota_x(a(x))
\]
and a path $G_{Y'} : [0,1]\to \U(A(Y'))$ such that $G_{Y'}(0) = \pi_{Y'}(b_x), G_{Y'}(1) = \pi_{Y'}(a)$, and $\pi^{Y'}_x\circ G_{Y'}$ is a constant path $a(x)$ in $\U(A(x))$ (the last statement follows from the proof of \cref{lem:local_injective}). Furthermore, $f:[0,1]\to \U(M_2(A(Y'))$ given by $f(s) := H(1-s,1)$ is a path such that $f(0) = \iota_{Y'}\circ \pi_{Y'}(b_x)$ and $f(1) = \iota_{Y'}\circ \pi_{Y'}(a)$. Also,
\[
\eta^{Y'}_x \circ f(s) = \iota_x(a(x)) = \eta^{Y'}_x\circ \iota_{Y'}\circ G_{Y'}(s) \quad\forall s\in [0,1]
\]
Thus, by \cref{lem:path_homotopy} applied to $f$ and $\iota_{Y'}\circ G_{Y'}$, there is a closed neighbourhood $Y$ of $x$ such that $Y\subset Y'$ and $\eta^{Y'}_Y\circ \iota_{Y'}\circ G_{Y'}\sim_h \eta^{Y'}_Y\circ f$ in $\U(M_2(A(Y))$. But $\eta_Y\circ F\sim_h (\eta_Y\circ \iota_A\circ G_x)\bullet (\eta^{Y'}_Y\circ f)$, so $\eta_Y\circ F\sim_h \iota_Y\circ L_Y$ where $L_Y: [0,1]\to \U(A(Y))$ is given by
\[
L_Y := (\pi_Y\circ G_x)\bullet (\pi^{Y'}_Y\circ G_{Y'})
\]
and $L_Y$ satisfies the required conditions.
\end{proof}

\begin{rem}\label{rem:reorganize_cover}
We are now in a position to prove \cref{thm:mainthm}, but first, we need one important fact, which allows us to use induction: If $X$ is a finite dimensional compact metric space, then covering dimension agrees with the small inductive dimension \cite[Theorem 1.7.7]{engelking}. Therefore, by \cite[Theorem 1.1.6]{engelking}, $X$ has an open cover $\mathcal{B}$ such that, for each $U \in \mathcal{B}$,
\[
\dim(\partial U)\leq \dim(X) - 1
\]
Now suppose $\{U_1,U_2,\ldots, U_m\}$ is an open cover of $X$ such that $\dim(\partial U_i)\leq \dim(X)-1$ for $1\leq i\leq m$, we define sets $\{V_i : 1\leq i\leq m\}$ inductively by
\[
V_1 := \overline{U_1}, \text{ and } V_k := \overline{U_k\setminus \left( \bigcup_{i<k} U_i\right)} \text{ for } k>1
\]
and subsets $\{W_j : 1\leq j\leq m-1\}$ by
\[
W_j := \left(\bigcup_{i=1}^j V_i\right)\cap V_{j+1}
\]
It is easy to see that $W_j \subset \bigcup_{i=1}^j \partial U_i$, so by \cite[Theorem 1.5.3]{engelking}, $\dim(W_j) \leq \dim(X)-1$ for all $1\leq j\leq m-1$
\end{rem}

\begin{proof}[Proof of \cref{thm:mainthm}]
Let $X$ be a compact metric space of finite covering dimension, and let $A$ be a continuous $C(X)$-algebra, each of whose fibers are $K$-stable. We wish to show that $A$ is $K$-stable. As in the proof of \cref{thm:dim_zero_case}, it suffices to show that the map
\[
\iota_{\ast} : \pi_0(\U(A))\to \pi_0(\U(M_2(A)))
\]
is bijective. The proof is by induction on $\dim(X)$, so by \cref{thm:dim_zero_case}, we assume that $\dim(X)\geq 1$, and that $A(Y)$ is $K$-stable for any closed subset $Y$ of $X$ such that $\dim(Y)\leq \dim(X)- 1$. \\

For injectivity, suppose $a\in \U(A)$ such that $\iota(a) \sim 0$ in $\U(M_2(A))$. Let $F:[0,1]\to \U(M_2(A))$ be a path such that $F(0) = \iota(a)$ and $F(1) = 0$. For $x\in X$, by \cref{lem:local_path}, there is a closed neighbourhood $Y_x$ of $x$ and a path $L_{Y_x} : [0,1]\to \U(A(Y_x))$ such that
\[
L_{Y_x}(0) = \pi_{Y_x}(a), \quad L_{Y_x}(1) = 0
\]
and $\iota_{Y_x}\circ L_{Y_x}$ is path homotopic to $\pi_{Y_x}\circ F$ in $\U(A(Y_x))$. By \cref{rem:reorganize_cover}, we may choose $Y_x$ to be the closure of a basic open set $U_x$ such that $\dim(\partial U_x)\leq \dim(X)-1$. Since $X$ is compact, we may choose a finite subcover $\{U_1,U_2,\ldots, U_m\}$. Now define $\{V_1,V_2,\ldots, V_m\}$ and $\{W_1,W_2,\ldots, W_{m-1}\}$ as in \cref{rem:reorganize_cover}. We observe that each $V_i$ is a closed set such that $\pi_{V_i}(a)\sim 0$ in $\U(A(V_i))$ since $V_i \subset \overline{U_i}$ for all $1\leq i\leq m$. \\

Note that $W_1 = V_1\cap V_2$, and $\dim(W_1)\leq \dim(X) - 1$. By induction hypothesis, $A(W_1)$ is $K$-stable. Let $H_i:[0,1]\to \U(A(V_i)), i=1,2$ be paths such that $H_i(0) = \pi_{V_i}(a), H_i(1) = 0$, and $\iota_{V_i}\circ H_i \sim_h \eta_{V_i}\circ F$. Let $S:[0,1]\to \U(A(W_1))$ be the path
\[
S(t) := \pi^{V_1}_{W_1}(H_1(t))\cdot \pi^{V_2}_{W_1}(H_2(t))^{\ast}
\]
Note that $S(0) = S(1) = 0$, so $S$ is a loop, and 
\[
\iota_{W_1}\circ S = (\eta^{V_1}_{W_1}\circ \iota_{V_1}\circ H_1)\cdot (\eta^{V_2}_{W_1}\circ \iota_{V_2}\circ H_2^{\ast}) \sim_h \eta_{W_1}\circ F\cdot (\eta_{W_1}\circ F)^{\ast} = 0
\]
Hence, $\iota_{W_1}\circ S$ is null-homotopic. Since the map $(\iota_{W_1})_{\ast} : \pi_1(\U(A(W_1)))\to \pi_1(\U(M_2(A(W_1))))$ is injective, it follows that $S$ is null-homotopic in $\U(A(W_1))$. Since the map $\pi^{V_2}_{W_1} : \U_0(A(V_2))\to \U_0(A(W_1))$ is a fibration, it follows that $S$ has a lift $f:[0,1]\to \U(A(V_2))$ such that $f(0) = f(1) = 0$ and $\pi^{V_2}_{W_1}\circ f = S$. Define $H_2' : [0,1]\to \U(A(V_2))$ be defined by
\[
H_2'(t) := f(t)\cdot H_2(t)
\]
Then $H_2'(0) = \pi_{V_2}(a), H_2'(1) = 0$, and
\[
\pi^{V_1}_{W_1}\circ H_1 = \pi^{V_2}_{W_1}\circ H_2'
\]
By \cref{lem:cx_algebra_pullback}, $A(V_1\cup V_2)$ is a pullback
\[
\xymatrix{
A(V_1\cup V_2)\ar[r]\ar[d] & A(V_1)\ar[d] \\
A(V_2)\ar[r] & A(W_1)
}
\]
so we obtain a path $H : [0,1]\to \U(A(V_1\cup V_2))$ such that $\pi^{V_1\cup V_2}_{V_1}\circ H = H_1$ and $\pi^{V_1\cup V_2}_{V_2}\circ H = H_2'$. In particular, $\pi^{V_1\cup V_2}_{V_1}\circ H(0) = \pi_{V_1}(a)$ and $\pi^{V_1\cup V_2}_{V_2}\circ H(0) = \pi_{V_2}(a)$ so $H(0) = \pi_{V_1\cup V_2}(a)$. Similarly, $H(1) = 0$, so that $\pi_{V_1\cup V_2}(a)\sim 0$ in $\U(A(V_1\cup V_2))$. \\

Now observe that $W_2 = (V_1\cup V_2)\cap V_3$, and $\dim(W_2)\leq \dim(X)-1$. Replacing $V_1$ by $V_1\cup V_2$, and $V_2$ by $V_3$ in the earlier argument, we may repeat the earlier procedure. By induction on the number of elements in the finite subcover, we see that $a\sim 0$ in $\U(A)$. This completes the proof of injectivity of $\iota_{\ast}$.\\

Now consider the surjectivity of $\iota_{\ast}$: Fix $u \in \U(M_2(A))$, and we wish to show that there is a quasi-unitary $\omega \in \U(A)$ such that $u\sim \iota(\omega)$. So fix $x\in X$. Then by $K$-stability of $A(x)$, there exists $f_x \in \U(A(x))$ such that $\eta_x(u) \sim \iota_x(f_x)$. As in the proof of \cref{thm:dim_zero_case}, there is a closed neighbourhood $Y_x$ of $x$ and a quasi-unitary $d_x \in \U(A(Y_x))$ such that
\[
\eta_{Y_x}(u) \sim \iota_{Y_x}(d_x)
\]
As in the first part of the proof, we may reduce to the case where $X = V_1\cup V_2$, and there are quasi-unitaries $d_{V_1} \in \U(A(V_1)), d_{V_2} \in \U(A(V_2))$ such that
\[
\eta_{V_i}(u) \sim \iota_{V_i}(d_{V_i}) \text{ in } \U(M_2(A(V_i)), i=1,2
\]
and if $W := V_1\cap V_2$, then
\[
\dim(W)\leq \dim(X)-1
\]
Fix paths $H_i : [0,1]\to \U(M_2(A(V_i))$ such that $H_1(0) = \iota_{V_1}(d_{V_1})$ and $H_1(1) = \eta_{V_1}(u)$, $H_2(0)=\eta_{V_2}(u)$ and $H_2(1)=\iota_{V_2}(d_{V_2})$ and consider the path $F :[0,1]\to \U(M_2(A(W))$ given by
\[
F := (\eta^{V_1}_W\circ H_1) \bullet (\eta^{V_2}_W\circ H_2)
\]
Then $F(0) = \iota_W\circ \pi^{V_1}_W(d_{V_1})$ and $F(1) = \iota_W\circ \pi^{V_2}_W(d_{V_2})$. By induction, $A(W)$ is $K$-stable, so by \cref{lem:k_stable_def_retract}, there is a retraction $r_W : \U(M_2(A(W)) \to \U(A(W))$. Define $H := r_W\circ F$, then $H : [0,1]\to \U(A(W))$ is a path such that
\[
H(0) = \pi^{V_1}_W(d_{V_1}), \text{ and } H(1) = \pi^{V_2}_W(d_{V_2})
\]
The map $\pi^{V_2}_W : \U_0(A(V_2)) \to \U_0(A(W))$ is a fibration, so there is a path $H' : [0,1]\to \U(A(V_2))$ such that
\[
H'(1) = d_{V_2}, \text{ and } \pi^{V_2}_W\circ H' = H
\]
Define $e_{V_2} := H'(0)$ so that
\[
\pi^{V_2}_W(e_{V_2}) = \pi^{V_1}_W(d_{V_1})
\]
By \cref{lem:cx_algebra_pullback}, $A = A(V_1\cup V_2)$ is a pullback
\[
\xymatrix{
A\ar[r]^{\pi_{V_1}}\ar[d]_{\pi_{V_2}} & A(V_1)\ar[d]^{\pi^{V_1}_W} \\
A(V_2)\ar[r]^{\pi^{V_2}_W} & A(W)
}
\]
so that $\omega := (d_{V_1}, e_{V_2})$ defines a quasi-unitary in $A$. We claim that $\iota(\omega)\sim u$ in $\U(A)$. To this end, define $H_3 : [0,1]\to \U(M_2(A(V_2))$ by
\[
H_3 := H_2\bullet (\iota_{V_2}\circ \overline{H'})
\]
then $H_3(0) = \eta_{V_2}(u)$ and $H_3(1) = \iota_{V_2}(e_{V_2}) = \iota_{V_2}(\pi_{V_2}(\omega)) = \eta_{V_2}(\iota_{V_2}(\omega))$. So if $S : [0,1]\to \U(M_2(A(W))$ is given by
\[
S(t) := \eta^{V_2}_W(H_3(t))\cdot \eta^{V_1}_W(H_1(1-t))^{\ast}
\]
Then $S$ is a path with $S(0) = \eta_W(u)\cdot \eta_W(u)^{\ast} = 0$ and $S(1) = \iota^{V_2}_W(e_{V_2})\cdot \eta^{V_1}_W(\iota_{V_1}(d_{V_1})) = 0$. Furthermore,
\begin{equation*}
\begin{split}
\eta^{V_2}_W\circ H_3 &= (\eta^{V_2}_W\circ H_2)\bullet (\eta^{V_2}_W\circ \iota_{V_2}\circ \overline{H'}) = (\eta^{V_2}_W\circ H_2)\bullet (\iota_W\circ \pi^{V_2}_W\circ \overline{H'}) \\
&= (\eta^{V_2}_W\circ H_2)\bullet (\iota_W\circ \overline{H}) = (\eta^{V_2}_W\circ H_2)\bullet (\iota_W\circ r_W\circ \overline{F}) \\
&\sim_h (\eta^{V_2}_W\circ H_2)\bullet \overline{F}
\end{split}
\end{equation*}
Hence,
\begin{equation*}
\begin{split}
S &\sim_h [(\eta^{V_2}_W\circ H_2)\bullet \overline{F}]\cdot (\eta^{V_1}_W\circ \overline{H_1})^{\ast} \\
&= [(\eta^{V_2}_W\circ H_2)\bullet (\eta^{V_2}_W\circ \overline{H_2})\bullet (\eta^{V_1}_W\circ \overline{H_1})]\cdot (\eta^{V_1}_W\circ \overline{H_1})^{\ast} \\
&\sim_h (\eta^{V_1}_W\circ \overline{H_1})\cdot (\eta^{V_1}_W\circ \overline{H_1})^{\ast} = 0
\end{split}
\end{equation*}
Hence, $S$ is null-homotopic. Once again, the map $\eta^{V_1}_W : \U_0(M_2(A(V_1)) \to \U_0(M_2(A(W))$ is a fibration, so there is a loop $f:[0,1]\to \U(M_2(A(V_1))$ such that $f(0) = f(1) = 0$ and $\eta^{V_1}_W \circ f = S$. Define $H_4 : [0,1] \to \U(M_2(A(V_1))$ by
\[
H_4(t) := f(t)\cdot H_1(1-t)
\]
Then $H_4(0) = \eta_{V_1}(u), H_4(1) = \iota_{V_1}(d_{V_1}) = \eta_{V_1}(\iota(\omega))$. Finally, by construction
\[
\eta^{V_1}_W\circ H_4 = \eta^{V_2}_W\circ H_3
\]
Therefore, the pair $(H_3,H_4)$ defines a path in $\U(M_2(A))$ connecting $u$ to $\iota(\omega)$. This concludes the proof that $\iota_{\ast}$ is surjective.
\end{proof}

\subsection*{Acknowledgements}
The first named author is supported by UGC Junior Research Fellowship No. 1229 and the second named author was supported by SERB Grant YSS/2015/001060.

\bibliographystyle{plain}
\bibliography{k_stable_fields}

\end{document}